\documentclass[12pt]{amsart}

\usepackage[latin1,utf8]{inputenc}
\usepackage[T1]{fontenc} 
\usepackage{textcomp} 
\usepackage[english]{babel}
\usepackage[autolanguage]{numprint} 
\usepackage{hyperref} 
\usepackage{graphicx} 
\usepackage{verbatim} 
\usepackage[all]{xy}
\usepackage{float}

\usepackage{lmodern}
\usepackage{mathrsfs}
\usepackage{amsmath,amssymb,amsthm} 
\usepackage{enumitem}

\newtheorem{theorem}{Theorem}[section]

\newtheorem{lemma}[theorem]{Lemma}

\newtheorem{corollary}[theorem]{Corollary}

\theoremstyle{remark}

\newtheorem*{remark}{Remark}

\newcommand{\bN}{\mathbb{N}}
\newcommand{\bP}{\mathbb{P}}
\newcommand{\bQ}{\mathbb{Q}}
\newcommand{\bR}{\mathbb{R}}

\newcommand{\bxi}{\boldsymbol{\xi}}
\newcommand{\balpha}{\boldsymbol{\alpha}}
\newcommand{\bbeta}{\boldsymbol{\beta}}

\newcommand{\bZ}{\mathbb{Z}}

\newcommand{\dist}{\mathrm{dist}}
\newcommand{\DD}{\mathcal{D}}

\newcommand{\lambdahat}{\hat{\lambda}}

\newcommand{\norm}[1]{\|#1\|}

\renewcommand\phi{\varphi}

\renewcommand\theta{\vartheta}

\newcommand{\ux}{\mathbf{x}}
\newcommand{\uy}{\mathbf{y}}

\newcommand{\hlambda}{\widehat{\lambda}}

\newcommand{\Qspan}[1]{\langle\,#1\rangle_\bQ}
\newcommand{\Rspan}[1]{\langle\,#1\rangle_\bR}
\newcommand{\Zli}{Z^{\mathrm{li}}}

\def\bx{{\bf x}}
\def\by{{\bf y}}
\def\bz{{\bf z}}

\def\balpha{\boldsymbol{\alpha}}

\def\Grass{G_{k,n}} 

\newcommand{\Ack}{{
  \footnotesize

  \textbf{Acknowledgements}: I am very grateful to Damien Roy for giving me a lot of
  feedback on this work. I also thank the anonymous referee for his useful comments. This work is partially supported by NSERC.
}}

\begin{document}

\baselineskip=17pt 

\title{A class of maximally singular sets for rational approximation}

\author{Anthony Po\"els}
\address{
   D\'epartement de Math\'ematiques\\
   Universit\'e d'Ottawa\\
   150 Louis Pasteur\\
   Ottawa, Ontario K1N 6N5, Canada}
\email{anthony.poels@uottawa.ca}



\begin{abstract}
\noindent We say that a subset of $\bP^n(\bR)$ is maximally singular if its contains points with $\bQ$-linearly independent homogenous coordinates whose uniform exponent of simultaneous rational approximation is equal to $1$, the maximal possible value. In this paper, we give a criterion which provides many such sets including Grassmannians. We also recover a result of the author and Roy about a class of quadratic hypersurfaces.

\end{abstract}

\maketitle

\section{Introduction}
\label{section: intro}

A basic problem in Diophantine approximation is given a subset $Z$ of $\bP^n(\bR)$ to compute
\[
    \lambdahat(Z):=\sup\{\lambdahat(\xi)\,;\,\xi\in \Zli\} \in [1/n,1],
\]
where $\Zli$ stands for the set of elements of $Z$ having representatives with $\bQ$-linearly independent coordinates in $\bR^{n+1}$, and $\hlambda(\xi)$ is the so-called uniform exponent of simultaneous rational approximation to $\xi$ (see Section~\ref{section: notation and main result} for the definition). We say that $Z$ is \textsl{maximally singular} if $\hlambda(Z)=1$. Diophantine approximation in a projective setting is not a new topic, see for example~\cite{choi1999diophantine}, or \cite{ghosh2016projective,fishman2014intrinsic} for more recent results.

A famous example is the Veronese curve $Z=\{(1:\xi:\xi^2:\cdots:\xi^n)\, ; \, \xi\in\bR\}\subseteq\bP^n(\bR)$ for which the computation of $\hlambda(Z)$, open for $n\geq 3$, would have implications on problems of algebraic approximation (see \cite{davenport1969approximation}). The case $n=2$ is treated in \cite{roy2004approximation}. More generally, Roy computes in \cite{roy2013conics} this exponent for conics in $\bP^2(\bR)$. These results are extended by the author and Roy in \cite{poelsroy2019} to the case of quadratic hypersurface $Z$ of $\bP^n(\bR)$ defined over $\bQ$. We show that, if $\Zli$ is not empty, then the exponent $\hlambda(Z)$ is completely determined by $n$ and the Witt index of the quadratic form defining $Z$ \cite[Theorem~1.1]{poelsroy2019}. In particular, if this index is at least $2$, then $\hlambda(Z) = 1$ and there are uncountably many $\xi\in\Zli$ such that $\hlambda(\xi) = 1$.

Our main result presented in Section~\ref{section: notation and main result} gives a geometric condition under which a closed set $Z\subseteq\bP^n(\bR)$ (for the topology induced by the projective distance) satisfies $\hlambda(\xi) = 1$ for an uncountable set of points $\xi\in \Zli$, and thus $\hlambda(Z) = 1$. It allows us to recover the above mentionned result of \cite{poelsroy2019} concerning quadratic hypersurfaces of Witt index $\geq 2$ (see Section~\ref{section: hypersurface q}). It also admits the following consequence.

\begin{theorem}
    \label{thm: image map k-linear}
    Let $k,n,N$ be positive integers with $N\geq 2$, let $\phi: (\bR^n)^k \rightarrow \bR^{N+1}$ be a $k$-linear map defined over $\bQ$ such that $\phi((\bQ^n)^k)$ spans the whole space $\bR^{N+1}$ over $\bR$, let $Z$ denote the topological closure in $\bP^N(\bR)$ of the projectivization of $\phi((\bR^n)^k)$, then there are uncountably many $\xi\in Z^{\textrm{li}}$ such that $\hlambda(\xi) = 1$, and so $\hlambda(Z) = 1$.
\end{theorem}

Examples of such sets $Z$ are the Grassmannians $\Grass$ of $k$-dimensional subspaces of $\bR^n$ as well as the subset of $\bR[x_1,\dots,x_n]_k$ consisting of homogenous polynomials of degree $k$ in $n$ variables which factor as a product of $k$ linear forms over $\bR$ (see Section~\ref{section: grassmanian} for details). Note that, in general, $\phi((\bR^n)^k)$ need not be a closed subset of $\bR^{N+1}$  as an example of Bernau and Wojciechowski \cite{bernau1996images} shows. So its projectivization may not be closed.



\section{Main result and notation}
\label{section: notation and main result}

Let $n$ be an integer $\geq 1$. For each set $Z\subseteq\bP^n(\bR)$ (resp. $S\subseteq\bR^{n+1}$) we write $Z(\bQ) := Z\cap \bP^n(\bQ)$ (resp. $S(\bQ) = S\cap\bQ^{n+1}$). We denote by $[\ux]$ the class in $\bP^n(\bR)$ of a non-zero point $\ux$ of $\bR^{n+1}$.  Given non-zero points $\ux,\uy\in\bR^{n+1}$, we recall that the \emph{projective distance} between $[\ux]$ and $[\uy]$ is
\[
 \dist([\ux],[\uy])
  :=\dist(\ux,\uy)
  :=\frac{\norm{\ux\wedge\uy}}{\norm{\ux}\,\norm{\uy}}.
\]

Here, $\norm{\cdot}$ denotes the Euclidean norm. Let $\xi\in\bP^n(\bR)$ and let $\bxi\in\bR^{n+1}$ be a representative of $\xi$ so that $\xi=[\bxi]$. Following the notation of \cite{poelsroy2019}, we set
\[
 D_\xi(\ux)
   := \frac{\norm{\ux\wedge\bxi}}{\norm{\bxi}}
   = \norm{\ux}\dist(\xi,[\ux])
\]
for each non-zero $\ux\in\bZ^{n+1}$. Then, for each $X\ge 1$, we define
\[
 \DD_\xi(X)
  :=\min\left\{ D_\xi(\ux)\,;\,\ux\in\bZ^{n+1}\setminus\{0\}\ \text{and}\ \norm{\ux}\le X \right\}.
\]
The exponent of uniform approximation $\hlambda(\xi)$ alluded to in the introduction is defined as the supremum of all $\lambda\in\bR$ such that $\DD_\xi(X)\leq X^{-\lambda}$ for each sufficiently large $X$.

Our main result is the following.

\begin{theorem}
    \label{thm: ensemble general}
    Let $n\geq 2$ and let $Z$ be a closed subset of $\bP^n(\bR)$ (for the topology induced by the projective distance). Suppose that there exists a non-empty set $Z'\subseteq Z(\bQ)$ with the following property. For each proper projective subspace $H$ of $\bP^n(\bR)$ defined over $\bQ$ there is a function $s_H:Z' \rightarrow \bN$ such that, for each $x \in Z'$
    \begin{enumerate}[label=\rm(\roman*)]
        \item \label{item: thm general condition 1} if $s_H(x) = 0$, then $x \notin H$;
        \item \label{item: thm general condition 2} if $s_H(x) \geq 1$, then there exists a projective line $L$ defined over $\bQ$ such that
        \[
            x\in L(\bQ) \subseteq Z'\quad \textrm{and}\quad \#\{ y\in L(\bQ) \, ;\,  s_H(y)\geq s_H(x) \} < \infty.
        \]
    \end{enumerate}
    Finally let $\phi:[1,\infty)\rightarrow(0,1]$ be a monotonically decreasing function with $\lim_{X\rightarrow\infty}\phi(X)=0$ and $\lim_{X\rightarrow\infty}X\phi(X) = \infty$. Then there are uncountably many $\xi\in \Zli$ such that $\DD_{\xi}(X) \leq \phi(X)$ for all sufficiently large $X$.
\end{theorem}

Choosing $\phi=\log(3X)/X$ for $X\geq 1$ , we derive the following consequence.

\begin{corollary}
With the hypotheses and notation of the previous theorem, there are uncountably many $\xi\in \Zli$ such that $\hlambda(\xi)=1$, and so $\hlambda(Z)=1$.
\end{corollary}

\section{Proof of Theorem~\ref{thm: ensemble general}}
\label{section: proof thm ppal}

Assume that $Z\subseteq\bP^n(\bR)$ and $Z'\subseteq Z(\bQ)$ satisfy the hypotheses of Theorem~\ref{thm: ensemble general}.

\begin{lemma} 
\label{lemma: lemme 1}
Let $H$ be a proper projective subspace of $\bP^n(\bR)$ defined over $\bQ$ and let $s_H:Z'\rightarrow \bN$ be a function satisfying Conditions~\ref{item: thm general condition 1}-\ref{item: thm general condition 2} of Theorem~\ref{thm: ensemble general} for this choice of $H$. Then for any non-zero integer point $\bx\in\bZ^{n+1}$ such that $[\bx]\in Z'\cap H$ we have $s_H([\bx]) \geq 1$ and there exist infinitely many non-zero integer points $\by$ with $[\uy]\in Z'$ satisfying
\begin{equation}
    \label{eq: lemme 1}
    \dist(\ux,\uy) \leq \frac{C}{\norm{\by}}\quad \textrm{and}\quad s_H([\by]) < s_H([\bx]),
\end{equation}
for a constant $C = C(\bx,H)>0$ independent of $\by$.
\end{lemma}

\begin{proof}
Since $[\bx]\in H$, Condition~\ref{item: thm general condition 1} implies that $s_H([\bx]) \geq 1$. By Condition~\ref{item: thm general condition 2} of Theorem~\ref{thm: ensemble general}, there exists a non-zero integer point $\bz$ such that $\bP(\Qspan{\bx,\bz})\subseteq Z'$ and $s_H([\bz+b\bx]) \leq s_H([\bx])-1$ for all but finitely many $b\in\bZ$. Putting $\by:= \bz+b\bx$ we find
\[
    \dist(\bx,\by) = \frac{\norm{\bx\wedge\bz}}{\norm{\bx}\norm{\by}}
\]
with a numerator that is independent of the choice of $b$.
\end{proof}

\noindent\textbf{Proof of Theorem~\ref{thm: ensemble general}.} This is similar in many aspects to the proof of \cite[Proposition 10.4]{poelsroy2019}. Starting with a non-zero integer point $\ux_1$ with $[\ux_1]\in Z'$, we construct recursively a sequence $(\ux_i)_{i\geq 1}$ of non-zero points of $\bZ^{n+1}$ which satisfies the following properties.
When $i\geq 1$ we have
\begin{enumerate}[label=\rm(\alph*)]
    \item \label{construction condition 0} $[\ux_{i+1}]\in Z'$;
    \item \label{construction condition 2} $\norm{\ux_{i+1}} > \norm{\ux_{i}}$;
    \item \label{construction condition 3} $s_{H}([\ux_{i+1}]) < s_{H}([\ux_{i}])$ for the subspace $H=H_{i}$  of $\bP^n(\bR)$ given by
    \begin{equation*}
        H_i:= \bP\big(\Rspan{\ux_j,\dots,\ux_i}\big),
    \end{equation*}
where $j$ is the smallest index $\geq 1$ such that $H_i$ is a proper subspace of $\bP^n(\bR)$.

\end{enumerate}
When $i\geq 2$, we further ask that
\begin{enumerate}[label=\rm(\alph*)]
    \setcounter{enumi}{3}
    \item \label{construction condition 4}
     $\displaystyle \dist(\ux_{i+1},\ux_{i}) \leq \frac{C}{\norm{\ux_{i+1}}} \leq \frac{1}{3} \min\Big\{\frac{2\phi(\norm{\ux_{i+1}})}{\norm{\ux_{i}}}, \dist(\ux_{i},\ux_{i-1})\Big\}$,
     where $C=C(\ux_{i},H_{i})$ is the constant given by Lemma~\ref{lemma: lemme 1}.
\end{enumerate}

Suppose that $\ux_1,\dots,\ux_{i}$ are constructed for some $i\geq 1$. Then Lemma~\ref{lemma: lemme 1} provides a non-zero integer point $\ux_{i+1}$ of arbitrarily large norm satisfying Conditions~ \ref{construction condition 0} to~\ref{construction condition 3} as well as the left-hand side inequality of~\ref{construction condition 4}. If $i\geq 2$ the right-hand side inequality of~\ref{construction condition 4} is also fulfilled for $\norm{\ux_{i+1}}$ large enough since $\lim_{X\rightarrow\infty}X\phi(X) = \infty$.

The sequence $([\ux_i])_{i\ge 1}$ converges in $\bP^{n}(\bR)$ to a point $\xi$ with
\begin{equation}
 \label{proofv:prop:eq1}
 \dist(\xi,[\ux_{i}])
   \le \sum_{j=i}^\infty \dist(\ux_{j+1},\ux_{j})
   \le \dist(\ux_{i+1},\ux_{i}) \sum_{j=0}^\infty 3^{-j}
   = \frac{3}{2}\dist(\ux_{i+1},\ux_{i})
\end{equation}
for each $i\geq 1$. Moreover $\xi\in Z$ since $Z$ is a closed subset of $\bP^n(\bR)$ and $[\ux_i]\in Z'\subseteq Z$ for each $i\ge 1$. When $i\geq 2$, Condition~\ref{construction condition 4} combined with \eqref{proofv:prop:eq1} yields
\begin{equation}
 \label{proofv:prop:eq2}
 D_\xi(\ux_{i})
  = \norm{\ux_{i}}\dist(\xi,[\ux_{i}])
  \leq \frac{3}{2}\norm{\ux_{i}}\dist(\ux_{i+1},\ux_{i})
  \leq \phi(\norm{\ux_{i+1}}).
\end{equation}
In particular, we have $\lim_{i\to\infty} D_\xi(\ux_{i})=0$. Let $i_0,i$ be integers with $2\leq i_0\leq i$ such that $V:=\bP\big(\Rspan{\ux_{i_0},\dots,\ux_i}\big)$ is a strict subspace of $\bP^n(\bR)$. By definition of $H_i$ we have $V\subseteq H_i$. Moreover $[\ux_i]\in H_i$, so that $s_H([\ux_i]) \geq 1$, where $H= H_i$. Since $s_H \geq 0$, Condition~\ref{construction condition 3} implies that there exists a smallest integer $\ell > i$ such that $H_{\ell} \neq H_i$, and therefore $[\ux_{\ell}]\notin H_i$. In particular $[\ux_\ell]\notin V$ and this proves that $(\ux_i)_{i\geq i_0}$ spans $\bR^{n+1}$ for each $i_0 \geq 2$. By \cite[Lemma~6.2]{poelsroy2019} we deduce that $\xi\in\Zli$.

For each $X\geq \norm{\ux_2}$ we have $\norm{\ux_{i}}\leq X < \norm{\ux_{i+1}}$ for some $i\geq 2$ and using \eqref{proofv:prop:eq2} we obtain
\[
  \DD_\xi(X) \leq D_\xi(\ux_{i}) \le \phi(\norm{\ux_{i+1}})\le \varphi(X).
\]
Therefore the point $\xi\in\Zli$ has the required property. By varying the sequence $(\ux_i)_{i\geq 1}$, we obtain uncountably many such points as in the proof of \cite[Proposition 10.4]{poelsroy2019}.

\section{Proof of Theorem~\ref{thm: image map k-linear}}
\label{section: preuve thm map k-linear}

Let $k,n,N$ be positive integers with $N\geq2 $ and let $\phi: (\bR^n)^k \rightarrow \bR^{N+1}$ be a $k$-linear map defined over $\bQ$. We denote by
$Z_\bR$ (resp. $Z_\bQ$) the projectivization of $\phi((\bR^n)^k)$ (resp. $\phi((\bQ^n)^k)$) and by $S$ the set of points in $(\bQ^n)^k$ whose image via $\phi$ is non-zero. For each $\alpha,\beta\in Z_\bQ$, we define $m(\alpha, \beta)$ to be the largest integer $m\geq 0$ for which there exist $(\ux_1,\dots,\ux_k),(\uy_1,\dots,\uy_k)\in S$ satisfying
\begin{equation}
    \label{eq: def m(alpha,beta)}
    \alpha = [\phi(\ux_1,\dots,\ux_k)],\quad \beta = [\phi(\uy_1,\dots,\uy_k)],\quad \textrm{and} \quad  \#\{i\in [1,k]\,;\,\ux_i = \uy_i\} = m.
\end{equation}

\begin{lemma}
\label{lemma: lemme fonction k linear}
Suppose that $\phi((\bQ^n)^k)$ spans the whole space $\bR^{N+1}$. Let $H$ be a projective proper subspace of $\bP^N(\bR)$. For each $\alpha\in Z_\bQ$  the set $Z_\bQ\setminus H$ is non-empty and, upon defining
\[
    s_H(\alpha):= \min\big\{ k - m(\alpha,\beta) \,;\, \beta\in Z_\bQ \setminus H \big\},
\]
we have
\begin{enumerate}[label=\rm(\roman*)]
    \item \label{item lem pol condition 1} $s_H(\alpha) \geq 0$ with equality if and only if $\alpha\notin H$;
    \item \label{item lem pol condition 2} if $s_H(\alpha)\geq 1$, then there exists a projective line $L \subseteq \bP^N(\bR)$ defined over $\bQ$ such that
        \begin{equation}
            \label{eq: lemme pol}
            \alpha\in L(\bQ) \subseteq Z_\bQ \quad \textrm{and}\quad \{ y\in L(\bQ) \, ;\,  s_H(y)\geq s_H(\alpha) \} = \{\alpha \}.
        \end{equation}
\end{enumerate}
\end{lemma}

By Lemma~\ref{lemma: lemme fonction k linear} the topological closure $Z = \overline{Z_\bR}$ of $Z_\bR$ satisfies the hypotheses of Theorem~\ref{thm: ensemble general} with $Z'=Z_\bQ$. This, in turn, implies Theorem~\ref{thm: image map k-linear}.

\begin{proof}
Since by hypothesis $Z_\bQ$ generates the whole space $\bP^N(\bR)$ and since $H$ is a strict subspace, there exist points $\beta\in Z_\bQ\setminus H$. Assertion~\ref{item lem pol condition 1} is clear because $m(\alpha,\beta) = k$ if and only if $\alpha = \beta$.

To prove Assertion~\ref{item lem pol condition 2}, fix $\alpha\in Z_\bQ$ with $s_H(\alpha) \geq 1$ and choose $\beta\in Z_\bQ \setminus H$ such that
\[
    m:=m(\alpha,\beta) = k - s_H(\alpha) \leq k - 1
\]
is maximal. Choose $(\ux_1,\dots,\ux_k)$ (resp. $(\uy_1,\dots,\uy_k)$) in $S$ satisfying \eqref{eq: def m(alpha,beta)} with $\alpha,\beta$ and $m$ as above. Without lost of generality, we may assume that $\ux_k\neq \uy_k$. Then set
\begin{align*}
    \balpha &= \phi(\ux_1,\dots,\ux_k), & & \balpha' =  \phi(\ux_1,\dots,\ux_{k-1},\uy_k), \\
    \bbeta &= \phi(\uy_1,\dots,\uy_k), & & \bbeta' =  \phi(\uy_1,\dots,\uy_{k-1},\ux_k).
\end{align*}
By hypothesis, we have $\balpha,\bbeta\neq 0$, $\alpha = [\balpha]$ and $\beta = [\bbeta]$. For any $\lambda,\mu\in \bQ$ the points
\begin{equation}
    \label{eq inter: lem fnt k linear 1}
    \lambda\balpha + \mu\balpha' = \phi(\ux_1,\dots,\ux_{k-1},\lambda\ux_k+\mu\uy_k)\quad \textrm{and}\quad  \mu\bbeta + \lambda\bbeta' = \phi(\uy_1,\dots,\uy_{k-1},\lambda\ux_k+\mu\uy_k)
\end{equation}
belong to $\phi((\bQ^n)^k)$. Suppose first that $\balpha'$ is proportional to $\balpha$. There exist infinitely many $\mu\in\bQ^*$ such that $\balpha + \mu\balpha'\neq 0$ and $\mu\bbeta +\bbeta'\neq 0$. For those $\mu$, the formulas \eqref{eq inter: lem fnt k linear 1} yield
\[
    m(\alpha,[\mu\bbeta +\bbeta']) = m([\balpha + \mu\balpha'],[\mu\bbeta +\bbeta']) \geq m+1,
\]
which implies that $[\mu\bbeta + \bbeta']\in H$ by maximality of $m$. This is impossible since $[\bbeta] = \beta\notin H$. Hence $\balpha'$ is not proportional to $\balpha$ and
\[
    L:=\bP(\Rspan{\balpha,\balpha'})
\]
is a projective line of $\bR^N(\bR)$ defined over $\bQ$, satisfying $\alpha\in L(\bQ) \subseteq Z_\bQ$. To conclude, it remains to show that $L$ satisfies the second condition in \eqref{eq: lemme pol}. We first prove that for any $\lambda\in\bQ$ we have
\begin{equation}
\label{eq: proof thm 1.1 eq inter 1}
    \bbeta+\lambda\bbeta'\neq 0\quad \textrm{and}\quad [\bbeta+\lambda\bbeta']\notin H.
\end{equation}
It is true if $\bbeta' = 0$. If $\bbeta'\neq 0$, then $m(\alpha,[\bbeta']) \geq m+1$. This gives $[\bbeta']\in H$ by maximality of $m$, and~\eqref{eq: proof thm 1.1 eq inter 1} follows since $\beta = [\bbeta]\notin H$. Combining \eqref{eq inter: lem fnt k linear 1} and~\eqref{eq: proof thm 1.1 eq inter 1}, we get
\[
    m([\lambda\balpha+\balpha'],[\bbeta + \lambda\bbeta']) \geq m+1,
\]
which yields $s_H([\lambda\balpha+\balpha'])\leq k-m - 1 = s_H(\alpha) -1$.
\end{proof}

\section{Two examples}
\label{section: grassmanian}

In this section we give two examples of sets to which Theorem~\ref{thm: image map k-linear} applies.

\noindent\textbf{The sets $\Grass$.} Let $k,n$ be two integers with $1\leq k < n$ and $n\geq 3$. We define the Grassmannian $\Grass$ as the projectivization of the set
\[
    \{ \ux_1\wedge\dots \wedge\ux_k \, | \, \ux_1,\dots,\ux_k\in \bR^n \}
\]
inside $\bP(\bigwedge^k\bR^n)$. By identifying $\bigwedge^k\bR^n$ to $\bR^N$ via an ordering of the Plücker coordinates, where $N = \binom{n}{k}\geq 3$, the set $\Grass$ is identified to a subset $\bP(\bR^N)$. It is Zariski closed, thus closed. By construction it is the projectivization of the image of the $k$-linear map $\phi:(\bR^n)^k\rightarrow \bR^N$ defined by
\[
    \phi(\ux_1,\dots,\ux_k) = \ux_1\wedge\dots\wedge\ux_k,\quad  \ux_1,\dots,\ux_k\in \bR^n.
\]
It is easily seen that $\phi((\bQ)^n)$ spans $\bR^N$, so that the conditions of Lemma~\ref{lemma: lemme fonction k linear} are fulfilled and we get $\hlambda(\Grass) = 1$.

\noindent\textbf{The sets $H_{n,k}$.} Let $k,n$ be positive integers with $n\geq 2$ and $n+k\geq 4$. The vector space $\bR[x_1,\dots,x_n]_k$ of homogenous polynomials of degree $k$ in $n$ variables admits for basis the set of monomials
\[
    x_1^{\alpha_1}\cdots x_n^{\alpha_n},\quad \alpha_1+\dots + \alpha_n = k
\]
which identifies it to $\bR^N$, where $N=\binom{n+k-1}{k}\geq 3$. We denote by $H_{n,k}\subseteq \bP^N(\bR)$ the projectivization of the set
\begin{equation*}
    \big\{ \prod_{j=1}^k L_j\, | \; L_1,\dots, L_k\in \bR[x_1,\dots,x_n]_1 \big\} = \psi((\bR[x_1,\dots,x_n]_1)^k),
\end{equation*}
where the $k$-linear map $\psi: (\bR[x_1,\dots,x_n]_1)^k \rightarrow \bR[x_1,\dots,x_n]_k \cong \bR^N$ is defined by
\[
    \psi(L_1,\dots,L_k) = L_1\cdots L_k
\]
for any $L_1,\dots,L_k\in\bR[x_1,\dots,x_n]_1$. It can be shown that $H_{n,k}$ is closed. Since $\psi((\bQ[x_1,\dots,x_n]_1)^k)$ spans $\bR^N$, we can once again apply Lemma~\ref{lemma: lemme fonction k linear} to get $\hlambda(H_{n,k}) = 1$.

\section{Quadratic hypersurfaces of Witt index $> 1$}
\label{section: hypersurface q}

Let $n\geq 1$ be an integer. A quadratic hypersurface of $\bP^n(\bR)$ defined over $\bQ$ is a non-empty subset which is the set of zeros in $\bP^n(\bR)$ of an irreductible homogeneous polynomial $q$ of $\bQ[t_0,\dots,t_n]$ of degree $2$. The Witt index (over $\bQ$) $m$ of $q$ is the largest integer $m\geq 0$ such that $\bQ^{n+1}$ contains an orthogonal sum (with respect to the symmetric bilinear form associated to $q$) of $m$ hyperbolic planes for $q$.

The following result is part of \cite[Theorem~1.1]{poelsroy2019}.

\begin{theorem}[Poëls-Roy, 2019]
    \label{thm: quadric, Poels-Roy}
    Let $n\geq 3$ be an integer and let $Z$ be a quadratic hypersurface of $\bP^n(\bR)$ defined over $\bQ$, and let $m$ be the Witt index (over $\bQ$) of the quadratic form on $\bQ^{n+1}$ defining $Z$. If $m\geq 2$, then there are uncountably many $\xi\in\Zli$ such that $\hlambda(\xi) = 1$, and thus $\hlambda(Z) = 1$.
\end{theorem}

In this section we use our Theorem~\ref{thm: ensemble general} to give an alternative shorter proof of this statement. We choose $Z'= Z(\bQ)$ and for each proper projective subspace $H\subseteq \bP^n(\bR)$ defined over $\bQ$ and each $\alpha\in Z(\bQ)$ we set
\begin{equation}
    s_H(\alpha) =
    \left\{\begin{array}{l}
    \textrm{$0$ if $\alpha\notin H$}, \\
    \textrm{$1$ if $\alpha\in H$ and $\langle\alpha\rangle^\perp \neq H$}, \\
    \textrm{$2$ if $\alpha\in H$ and $\langle\alpha\rangle^\perp = H$}, \\
    \end{array}\right.
\end{equation}
where for any $\beta\in \bP^n(\bR)$ the set $\langle\beta\rangle^\perp$ denotes the orthogonal in $\bP^n(\bR)$ of $\beta$ (with respect to the symmetric bilinear form associated to $q$). This choice of $s_H$ satisfies Condition~\ref{item: thm general condition 1} of Theorem~\ref{thm: ensemble general}. If $s_H(\alpha) \geq 1$ we show that there exists a projective line $L$ defined over $\bQ$ with the following properties stronger than Condition~\ref{item: thm general condition 2} of Theorem~\ref{thm: ensemble general}:
\begin{enumerate}[label=\rm(\Roman*)]
    \item \label{item: preuve thm poels-roy 1} $\alpha \in L(\bQ) \subseteq Z(\bQ)$;
    \item \label{item: preuve thm poels-roy 2} $s_H(\beta) < s_H(\alpha)$ for any $\beta\in L(\bQ)\setminus\{\alpha\}$.
\end{enumerate}

If $s_H(\alpha) = 2$, we choose a projective line $L$ defined over $\bQ$ satisfying \ref{item: preuve thm poels-roy 1} (this exists since the Witt index $m$ of $q$ is $\geq 2$). We claim that $L$ also satisfies \ref{item: preuve thm poels-roy 2}. Indeed, since $H^\perp = \{\alpha\}$, we have $\langle\beta\rangle^\perp \neq \langle\alpha\rangle^\perp = H$ for any $\beta\neq \alpha$, and so \ref{item: preuve thm poels-roy 2} follows.

If $s_H(\alpha) = 1$, the situation is more complicated. Arguing as in the proof of \cite[Lemma~10.3]{poelsroy2019}, we note that there exists a zero $\beta$ of $q$ in $Z(\bQ) \cap \langle\alpha\rangle^\perp \setminus H$. Since $\alpha\in H$ and $\beta\notin H$, any point $\gamma \neq \alpha$ in the projective line $L\subseteq Z$ generated by $\alpha$ and $\beta$ does not belong to $H$ and thus satisfies $s_H(\gamma) = 0 < s_H(\alpha)$.

\begin{remark}
In the present paper and in~\cite{poelsroy2019}, we search for points of a projective (not necessarily non-singular) quadratic hypersurface $Z$ which are \textsl{very well approximated} by rational points of the ambient space. More precisely we look for points whose uniform exponent of rational simultaneous approximation is maximal. When the Witt index $m$ of $Z$ is $\leq 1$, we show in~\cite{poelsroy2019} that the best rational approximations to such points are necessarily outside $Z$. On the contrary, when $m>1$ as in Theorem~\ref{thm: quadric, Poels-Roy}, the best rational approximations are necessarily points of $Z$. The authors of~\cite{fishman2014intrinsic} consider the problem of simultaneous approximation to real points on a projective non-singular quadratic hypersurface $Z$ by rational points of $Z$, they call that \textsl{intrinsic approximation}. Moreover, their goal is in a sense opposite to ours; they show that all points of $Z$ admit intrinsic rational approximations to a certain precision, and compute the Hausdorff dimension of the set of \textsl{badly approximable} points. It is interesting to note that in~\cite{fishman2014intrinsic}, the rational Witt index (called $\bQ$--rank) also plays an important role.
\end{remark}


\Ack

\bibliographystyle{abbrv}

\footnotesize {


}

\end{document}